\documentclass{article}%

\usepackage{amsmath,amssymb,amsfonts}%
\usepackage{theorem}%
\usepackage{color}%
\usepackage{hyperref}%
\usepackage{bbm}%

\theoremstyle{change}%
\newtheorem{definition}{Definition:}[section]%
\newtheorem{proposition}[definition]{Proposition:}%
\newtheorem{theorem}[definition]{Theorem:}%
\newtheorem{lemma}[definition]{Lemma:}%
\newtheorem{corollary}[definition]{Corollary:}%
{\theorembodyfont{\rmfamily}\newtheorem{remark}[definition]{Remark:}}%
{\theorembodyfont{\rmfamily}\newtheorem{example}[definition]{Example:}}%

\newenvironment{proof}
  {{\bf Proof:}}
  {\qquad \hspace*{\fill} $\Box$}%
  
\newcommand{\fa}{\mathfrak{a}}%
\newcommand{\fg}{\mathfrak{g}}%
\newcommand{\fk}{\mathfrak{k}}%
\newcommand{\fn}{\mathfrak{n}}%
\newcommand{\fz}{\mathfrak{z}}%
\newcommand{\fh}{\mathfrak{h}}%

\newcommand{\Ad}{\operatorname{Ad}}%
\newcommand{\ad}{\operatorname{ad}}%
\newcommand{\id}{\operatorname{id}}
\newcommand{\inner}{\operatorname{int}}%
\newcommand{\rme}{\mathrm{e}}%

\newcommand{\CC}{\mathcal{C}}%
\newcommand{\UC}{\mathcal{U}}%
\newcommand{\DC}{\mathcal{D}}%
\newcommand{\XC}{\mathcal{X}}%
\newcommand{\AC}{\mathcal{A}}%
\newcommand{\N}{\mathbb{N}}%
\newcommand{\R}{\mathbb{R}}%
\newcommand{\C}{\mathbb{C}}%
  
\begin{document}

\title{Controllability of linear systems on solvable Lie groups}
\author{Adriano Da Silva\footnote{Supported by Fapesp grant $n^o$ 2013/19756-8}\\Instituto de Matem\'atica,\\
Universidade Estadual de Campinas\\
Cx. Postal 6065, 13.081-970 Campinas-SP, Brasil}
\date{\today}
\maketitle

{\bf Abstract. }Linear systems on Lie groups are a natural generalization of linear system on Euclidian spaces. For such systems, this paper studies controllabi\-lity by taking in consideration the eigenvalues of an associated derivation $\DC$. When the state space is a solvable connected Lie group, controllability of the system is guaranteed if the reachable set of the neutral element is open and the derivation $\DC$ has only pure imaginary eigenvalues. For bounded systems on nilpotent Lie groups such conditions are also necessary.



\section{Introduction}

Linear systems on Lie groups were introduced in \cite{VAJT} and \cite{VASM} as a natural gene\-ralization of the linear system on the Euclidean space

$$\dot{x}(t)=Ax(t)+Bu(t),\;\;\;A\in\R^{d\times d}, B\in\R^{d\times m}.$$

Controllability of such systems, with unrestricted control functions, was studied in \cite{VAJT}, \cite{VASM} and \cite{JPh} and null controllability in \cite{VAEK}.

For restricted control functions, the controllability of the above Euclidean system was studied in \cite{Sontag}. It was shown (Theorem 6 of \cite{Sontag}) that the system is controllable if, and only if, the matrix $A$ has only pure imaginary eigenvalues and if the pair $(A, B)$ is controllable, that is, if 
$$\mathrm{span}\{A^jb_i, i=1, \ldots, m, \;j=1, \ldots, d\}=\R^d,$$
where $b_i$, $i=1, \ldots, m$ are the columns of the matrix $B$. Moreover the controllability of the pair $(A, B)$ is equivalent to have that the reachable set from $0\in\R^d$ is open.

To a linear system on a connected Lie group $G$ we have associated a derivation $\DC$ of the Lie algebra $\fg$ of $G$. We show that if $G$ is a solvable connected Lie group a sufficient condition to obtain controllability of the linear system is that $\DC$ has only pure imaginary eigenvalues and that the reachable set from the neutral element is open in $G$, which generalizes the result in $\R^d$ since in this case the derivation $\DC$ coincides with the linear map induced by the matrix $A$. If the Lie group $G$ is nilpotent and the control functions are restricted we show that the the above conditions are also necessary.

The structure of the paper is the following: Section 2 gives the preliminaries that we need, such as the definitions of linear vector fields and linear systems. We introduce here several Lie subalgebras and Lie subgroups induced by the derivation $\DC$ associated with a linear vector field and prove some important results concerning them. In the Section 3 we show that there is a close relation between the reachable set from the neutral element and the subgroups introduced in Section 2. It is shown that if the reachable set from the neutral element is open, then it contains the subgroup associated with the generalized eigenspaces of $\DC$ whose eigenvalues have nonnegative real part. In Section 4 we show our main results on controllability using the results of section 3. We also prove that for linear systems on nilpotent Lie groups with bounded control functions, the conditions that we found are also necessary.

{\bf Notation:} The set of $\CC^{\infty}$ vector fields of a manifold $M$ will be denoted as $X(M)$. Let $\varrho$ be a metric in $M$  and let $x\in M$ and $\varepsilon>0$. The $\varepsilon$-ball centered at $x$ and its closure are given, respectively, by 
$$B_{\varepsilon}(x):=\{z\in M; \;\varrho(x, z)<\varepsilon\}\hspace{.3cm}\mbox{ and }\hspace{.3cm}\bar{B}_{\varepsilon}(x):=\{z\in M; \;\varrho(x, z)\leq\varepsilon\}.$$ 

A Lie group $G$ with Lie algebra $\fg$ will be always considered to be a $\CC^{\infty}$ manifold. For $X, Y\in\fg$ the {\bf Baker-Campbell-Hausdorff formula} ({\bf BCH}) is given by
$$\exp(X)\exp(Y)=\exp(S(X, Y))$$
where $S(X, Y)$ is a series depending on $X$, $Y$ and its brackets. The first terms are given by
\begin{equation}
\label{campbell}
S(X, Y)=X+Y+\frac{1}{2}[X, Y]+\frac{1}{12}[[X, Y], Y]-\frac{1}{12}[[X, Y], X]+\cdots
\end{equation}
where the next terms depend on the brackets of four or more elements. This series converges for $X$ and $Y$ small enough. In particular, when the group $G$ is nilpotent, the above series is finite for all $X, Y\in\fg$. 

\section{Preliminaries}

In this Section we will introduce the concepts of linear vector field and linear systems on Lie groups. We will also state several results about the subgroups induced by the derivation $\DC$ associated to a linear vector field.

\subsection{Linear vector fields and Lie group decompositions}

Let $G$ be a connected Lie group with Lie algebra $\fg$. For a given derivation $\DC:\fg\rightarrow\fg$ we consider the generalized eigenspaces associated to $\mathcal{D}$ given by 
$$\mathfrak{g}_{\alpha}=\{X\in\mathfrak{g}; (\mathcal{D}-\alpha)^nX=0 \mbox{ for some }n\geq 1\}$$
where $\alpha$ is an eigenvalue of $\mathcal{D}$. We can decompose $\mathfrak{g}$ as
$$\fg=\fg^+\oplus\fg^-\oplus\fg^0$$
where 
$$\fg^+=\bigoplus_{\alpha; \,\mathrm{Re}(\alpha)> 0}\fg_{\alpha},\hspace{1cm}\fg^-=\bigoplus_{\alpha; \,\mathrm{Re}(\alpha)< 0}\fg_{\alpha}\hspace{1cm} \mbox{ and }\hspace{1cm}\fg^0=\bigoplus_{\alpha; \,\mathrm{Re}(\alpha)=0}\fg_{\alpha}.$$

The next proposition shows that the vector spaces $\fg^+$, $\fg^-$ and $\fg^0$ are Lie algebras and that $\fg^+$, $\fg^-$ are actually nilpotent. The proof can be found in \cite{SM1}, Proposition 3.1.

\begin{proposition}
\label{der1}
Let $\mathcal{D}:\mathfrak{g}\rightarrow\mathfrak{g}$ a derivation of the Lie algebra $\mathfrak{g}$ of finite dimension over a closed field. Consider the decomposition
$$\mathfrak{g}=\bigoplus_{\alpha}\mathfrak{g}_{\alpha}$$
where $\mathfrak{g}_{\alpha}$ is the generalized eigenspace associated to the eigenvalue $\alpha$. Then
\begin{equation}
\label{zerotrace}
[\mathfrak{g}_{\alpha}, \mathfrak{g}_{\beta}]\subset\mathfrak{g}_{\alpha+\beta},
\end{equation}
with $\mathfrak{g}_{\alpha+\beta}=0$ in case ${\alpha+\beta}$ is not an eigenvalue of $\mathcal{D}$.
\end{proposition}

\begin{remark}
\label{remarkrealcase}
If $\fg$ is a real Lie algebra, we can consider its complexification $\fg_{\C}$. Since the elements of $\fg_{\C}$ are of the form $X=\sum a_iX_i$ with $a_i\in\C$, $X_i\in\fg$, we can extend $\DC$ by linearity to a derivation $\DC_{\C}$ in $\fg_{\C}$. It is not hard to show using the properties of $\C$ that $\DC$ and $\DC_{\C}$ have the same eigenvalues and that
\begin{equation}
\label{eigenspace}
(\fg_{\alpha})_{\C}=(\fg_{\C})_{\alpha}
\end{equation}
where $\fg_{\alpha}$ and $(\fg_{\C})_{\alpha}$ are the generalized eigenspace associated respectively with $\DC$ and $\DC_{\C}$. That  implies that the above proposition is also valid when $\fg$ is a real Lie algebra.
\end{remark}

We will denote by $G^+$, $G^-$, $G^0$, $G^{+, 0}$ and $G^{-, 0}$ the connected Lie subgroups of $G$ with Lie algebras $\fg^+$, $\fg^-$, $\fg^0$, $\fg^{+, 0}:=\fg^+\oplus\fg^0$ and $\fg^{-, 0}:=\fg^-\oplus\fg^0$, respectively. 

\begin{lemma}
\label{projection}
Let $\pi_*:\fg\rightarrow\fh$ be a surjective homomorphism of Lie algebras. If $\DC_1$ and $\DC_2$ are derivations of $\fg$ and $\fh$, respectively, such that $\pi_*\circ\DC_1=\DC_2\circ\pi_*$ then
$$\pi_*(\fg^+)=\fh^+, \;\;\;\pi_*(\fg^-)=\fh^-\;\;\mbox{ and }\;\;\pi_*(\fg^0)=\fh^0.$$
Moreover, if $G$ and $H$ are connected Lie groups associated respectively with $\fg$ and $\fh$ and $\pi:G\rightarrow H$ is a surjective homomorphism such that $(d\pi)_e\circ\DC_1=\DC_2\circ(d\pi)_e$
then 
$$\pi(G^+)=H^+, \;\;\;\pi(G^-)=H^-\;\;\mbox{ and }\;\;\pi(G^0)=H^0.$$
\end{lemma}

\begin{proof}
Let $\alpha$ to be an eigenvalue of $\DC_1$. For any $X\in\fg_{\alpha}$ there exists $n\in\N$ such that $(\DC_1-\alpha)^nX=0$. Since $\pi_*\circ\DC_1=\DC_2\circ\pi_*$ we have that 
$$(\DC_2-\alpha)^n\pi_*(X)=\pi_*((\DC_1-\alpha)^nX)=0$$
which shows us that when $\pi_*(\fg_{\alpha})\neq \{0\}$ we have that $\alpha$ is necessarily an eigenvalue of $\DC_2$ and that $\pi_*(\fg_{\alpha})\subset\fh_{\alpha}$. Therefore $$\pi_*(\fg^+)\subset\fh^+, \;\;\;\pi_*(\fg^-)\subset\fh^-\;\;\mbox{ and }\;\;\pi_*(\fg^0)\subset\fh^0.$$
Also, since $\pi_*$ is surjective, we have that $\fh=\pi_*(\fg)=\pi_*(\fg^+)+\pi_*(\fg^-)+\pi_*(\fg^0)$ and consequently the above inclusions must be equalities proving the first assertion. The second one follows from the formula 
$$\pi(\exp_G X)=\exp_H(d\pi)_eX$$ 
and the fact that the Lie groups are connected.
\end{proof}

\bigskip

The {\bf normalizer} of $\fg$ is by definition the space
$$\eta:=\mathrm{norm}_{X(G)}(\fg):=\{F\in X(G); \,\mbox{ for all }Y\in\fg, \;\;[F, Y]\in\fg\}.$$

\begin{definition}
A vector field $\XC$ on $G$ is said to be {\bf linear} if it belongs to $\eta$ and if $\XC(e)=0$, where $e\in G$ stands for the neutral element of $G$.
\end{definition} 

The following result (Theorem 1 of \cite{JPh}) gives equivalent conditions for a vector field on $G$ to be linear.

\begin{theorem}
\label{propertieslinear}
Let $\XC$ be a vector field on a connected Lie group $G$. The fo\-llowing conditions are equivalent:
\begin{itemize}
\item[1.] $\XC$ is linear;
\item[2.] The flow of $\XC$ is a one parameter group of automorphisms of $G$, that is 
$$\varphi_t(gh)=\varphi_t(g)\varphi_t(h),$$
for any $t\in\R$, $g, h\in G$;
\item[3.] $\XC$ satisfies
\begin{equation}
\XC(gh)=(dL_g)_h\XC(h)+(dR_h)_g\XC(g), \;\;\mbox{ for all }\;\;g, h\in G.
\end{equation}
\end{itemize}
\end{theorem}

\bigskip
Let $(\varphi_t)_{t\in\mathbb{R}}$ denote the one parameter group of automorphisms of $G$ generated by the linear vector field $\mathcal{X}$. The second item on the above Theorem implies that $\XC$ is complete. Moreover, for any vector field $Y$ we have that
\begin{equation}
\label{colchete}
[\mathcal{X}, Y](e)=\frac{d}{dt}_{|t=0}(d\varphi_{-t})_{\varphi_t(e)}Y(\varphi_t(e))=\frac{d}{dt}_{|t=0}(d\varphi_{-t})_{e}Y(e)
\end{equation}
since $\varphi_t(e)=e$ for all $t\in\mathbb{R}$. Furthermore, if $Y$ is a right invariant vector field we have at any point $g\in G$ that
$$[\mathcal{X}, Y](g)=\frac{d}{dt}_{|t=0}(d\varphi_{-t})_{\varphi_t(g)}Y(\varphi_t(g))=\frac{d}{dt}_{|t=0}(d\varphi_{-t})_{\varphi_t(g)}(dR_{\varphi_t(g)})_eY(e)=$$
$$=\frac{d}{dt}_{|t=0}(dR_g)_e(d\varphi_{-t})_eY(e)=(dR_g)_e[\mathcal{X}, Y](e),$$
where in the fourth equality we used that $\varphi_{-t}\circ R_{\varphi_t(g)}=R_g\circ\varphi_{-t}$ (which follows direct from the property 2. above).

Then for a given linear vector field $\mathcal{X}$, one can associate the derivation $\mathcal{D}$ of $\mathfrak{g}$ defined as
$$\mathcal{D}Y=-[\mathcal{X}, Y](e), \mbox{ for all }Y\in\mathfrak{g}.$$
The minus sign in the above formula comes from the formula $[Ax, b]=-Ab$ in $\mathbb{R}^d$. It is also used in order to avoid a minus sign in the equality 
$$\varphi_t(\exp Y)=\exp(\mathrm{e}^{t\mathcal{D}}Y), \mbox{ for all }t\in\mathbb{R}, Y\in\mathfrak{g}$$ stated in the next proposition (Proposition 2 of \cite{JPh}).

\begin{proposition}
\label{derivativeonorigin}
For all $t\in\mathbb{R}$
$$(d\varphi_t)_e=\mathrm{e}^{t\mathcal{D}}$$
and since $\varphi_t$ is an automorphism of $G$ we have that
$$\varphi_t(\exp Y)=\exp(d\varphi_t)_eY=\exp(\mathrm{e}^{t\mathcal{D}}Y), \mbox{ for all }t\in\mathbb{R}, Y\in\mathfrak{g}.$$
\end{proposition}

\begin{remark}
Note that for the linear Euclidean case the formula $[Ax, b]=-Ab$ implies that the derivation $\mathcal{D}$ coincides with the linear map induced by $A$.
\end{remark}

\begin{remark}
\label{exponentiallystable}
Since $\DC|_{\fg^-}$ has only eigenvalues with negative real part we have that $0\in\fg^-$ is exponentially stable in positive time for the flow $\rme^{t\DC|_{\fg^-}}=\rme^{t\DC}|_{\fg^-}$, that is, there exists $\lambda, c>0$ such that $|\rme^{t\DC}X|\leq c^{-1}\rme^{-\lambda t}|X|$ for $t\geq 0$ and $X\in\fg^-$. The same is true for $\fg^+$ in negative time. 
\end{remark}
Let $\DC$ be the derivation associated with a linear vector field $\XC$ and $\varphi_t$ the flow of $\XC$. We say that a subspace $\fh\subset\fg$ is {\bf $\DC$-invariant} if $\DC(\fh)\subset\fh$. In the same way, we say that a subgroup $H$ of $G$ is {\bf $\varphi$-invariant} if $\varphi_t(H)=H$ for any $t\in \R$, where $\varphi_t$ is the flow associated with $\XC$. It is straightforward to see that if $H$ is a connected Lie subgroup with Lie algebra $\fh$ then $H$ is $\varphi$-invariant if, and only if, $\fh$ is $\DC$-invariant.

\begin{proposition}
\label{subgroups}
Let $\DC$ be the derivation associated with a linear vector field $\XC$. For the subgroups induced by  $\DC$ we have:
\begin{itemize}
\item[1.] $G^{+, 0}=G^+G^0=G^0G^+$ \;\;\;\;and\;\;\;\; $G^{-, 0}=G^-G^0=G^0G^-$;
\item[2.] $G^+\cap G^-=G^{+, 0}\cap G^-=G^{-, 0}\cap G^+=\{e\}$;
\item[3.] $G^{+, 0}\cap G^{-, 0}=G^0$;
\item[4.] All the above subgroups are closed in $G$;
\item[5.] If $G$ is solvable then 
$$G=G^{+, 0}G^-=G^{-, 0}G^+.$$ 
Moreover the fixed points of $\varphi$ are in $G^0$.
\end{itemize}
\end{proposition}

\begin{proof}
Since the above items have analogous statements we will prove just one assertion of each.

\begin{itemize}
\item[1.] By Proposition \ref{der1} and Remark \ref{remarkrealcase} we have that $[\fg^{+}, \fg^0]\subset\fg^+$ and consequently that $\fg^+$ is an ideal in $\fg^{+, 0}$. We have then that $G^+$ is a normal subgroup of $G^{+, 0}$ which already proves the second equality and also that $G^+G^0$ is a connected subgroup of $G$ with Lie algebra $\fg^{+, 0}$ which implies $G^{+, 0}=G^+G^0$ by unicity.

\item[2.] Note that the intersection $G^{+, 0}\cap G^-$ is a Lie group with Lie algebra $\fg^{+, 0}\cap\fg^-=\{0\}$ which implies that it is a discrete Lie subgroup of $G$. Let then $g\in G^{+, 0}\cap G^-$. Since $g\in G^-$ there exists $Z\in\fg^-$ such that $g=\exp(Z)$. By Remark \ref{exponentiallystable} and Proposition \ref{derivativeonorigin} we have that $\varphi_t(g)\rightarrow e$ as $t\rightarrow+\infty$. Since $\varphi_t(G^{+, 0}\cap G^-)=G^{+, 0}\cap G^-$ and $G^{+, 0}\cap G^-$ is discrete, there exists an open set $A\subset G$ such that $A\cap \left(G^{+, 0}\cap G^-\right)=\{e\}$ which implies that, for $t>0$ large enough, $\varphi_t(g)\in A\cap \left(G^{+, 0}\cap G^-\right)$ and consequently that $\varphi_t(g)=e$. Since $\varphi_t(e)=e$ for all $t\in\R$ we have $g=\varphi_{-t}(e)=e$ showing the result.

\item[3.] By item 1. above we already have that $G^0\subset G^{+, 0}\cap G^{-, 0}$. Let then $x\in G^{+, 0}\cap G^{-, 0}$. Since $x\in G^{+, 0}$ we have that $x=gg_0$ with $g\in G^+$ and $g_0\in G^0$ which implies that $G^+\ni g=xg_0^{-1}\in G^{-, 0}G^0=G^{-, 0}$ and by item 2. above we must have $g=e$. Then $x=g_0$ and consequently $ G^{+, 0}\cap G^{-, 0}=G^0$.

\item[4.] We will show that $G^{+, 0}$ is closed. Since $\fg=\fg^{+, 0}\oplus\fg^-$ there are open neighborhoods $0\in V\subset \fg^{+, 0}$, $0\in U\subset\fg^-$ and $e\in W\subset G$ such that the map $f:V\times U\rightarrow W$ defined by $f(X, Y)=\exp(X)\exp(Y)$ is a diffeomorphism and $W=\exp(V)\exp(U)$. Since $G^{+, 0}\cap G^-=\{e\}$ we have that $gW\cap G^{+, 0}=g\exp(V)$ for any $g\in G^{+, 0}$ which implies that, for any point $g\in G^{+, 0}$ there exists a neighborhood $W$ of $g$ in $G$, neighborhoods $V$ and $U$ of $0\in\fg^{+, 0}$ and of $0\in\fg^-$, respectively, and a diffeomorphism $f_g:=L_g\circ f:V\times U\rightarrow gW$ such that $f_g(V\times \{0\})=gW\cap G^{+, 0}$. Proposition B.1 of \cite{SM2} implies then that $G^{+, 0}$ is an embedded manifold of $G$ and consequently that it is a closed subgroup of $G$.

\item[5.] We will prove by induction on the dimension of $G$ that $G=G^{+, 0}G^-$ since the second one follows by applying the inversion and using item 1. above. If $\dim G=1$ then $G$ is abelian and the result is certainly true. Let us assume then that the result holds true for any solvable connected Lie group with dimension smaller than $n$ and let $G$ to be a solvable connected Lie group with dimension $n$. Since $G$ is solvable its Lie algebra $\fg$ is also solvable and we have the sequence of ideals of $\fg$ given by its derivative series
$$\fg=\fg^{(0)}\supset \fg^{(1)}\supset\ldots\supset\fg^{(k)}\supset\fg^{(k+1)}=\{0\}$$
where $\fg^{(i)}=[\fg^{(i-1)}, \fg^{(i-1)}]$ for $i=1, \ldots k$. Since $\DC$ is a derivation, each $\fg^{(i)}$ is $\DC$-invariant. Let $G^{(k)}$ to be the normal connected subgroup of $G$ with Lie algebra $\fg^{(k)}$. We have that $G^{(k)}$ is $\varphi$-invariant and since $\fg^{(k+1)}=\{0\}$ it is also an abelian Lie group. If we denote by $A$ the closure of $G^{(k)}$ we have that $A$ is a closed connected $\varphi$-invariant Lie subgroup of $G$ that is abelian and normal. Moreover, $\dim A\geq\dim G^{(k)}>0$ which implies that $H:=G/A$ is a solvable connected Lie group with $\dim H<\dim G$. Since $A$ is $\varphi$-invariance we have a well induced linear vector field on $H$ that is $\pi$-conjugated with $\XC$ in $G$ which implies that their associated derivations are $(d\pi)_e$ conjugated. By Lemma \ref{projection} and the induction hypothesis we get that $H=\pi(G^{+, 0}G^-)$ and consequently that 
$$G=G^{+, 0}G^-A=G^{+, 0}AG^-.$$
Since the Lie algebra $\fa$ of $A$ is abelian and $\DC$-invariant we also have that $A=A^{+, 0}A^-$ where the Lie algebras of $A^{+, 0}$ and $A^-$ are given respectively by 
$$\fa\cap\fg^{+, 0} \;\;\mbox{ and }\;\;\fa\cap\fg^-.$$
Therefore $A^{+, 0}\subset G^{+, 0}$ and  $A^0\subset G^0$ and consequently  
$$G=G^{+, 0}AG^-=G^{+, 0}A^{+, 0}A^-G^-\subset G^{+, 0}G^-\subset G$$
which proves the first part.

Let $g\in G$ to be a fixed point of $\varphi$ and $k_1\in G^+$, $k_2\in G^0$, $k_3\in G^-$ such that $g=k_1k_2k_3$. By item 2. and $\varphi$-invariance we have that $\varphi_t(g)=g$ if, and only if, $\varphi_t(k_i)=k_i$, $i=1, 2, 3$. Since $\varphi_t(k_1)\rightarrow e$ for $t\rightarrow-\infty$ we have that $\varphi_t(k_1)=k_1$ for any $t\in\R$ if, and only if, $k_1=e$. Analogously we have that $k_3=e$ which implies that $g=k_2\in G^0$ concluding the proof.
\end{itemize}
\end{proof}

The next proposition shows that if $H\subset G$ is a $\varphi$-invariant compact Lie subgroup, then $H\subset G^0$.

\begin{proposition}
\label{compact}
Let $H\subset G$ be a $\varphi$-invariant connected Lie subgroup. If $H$ is compact, then $H\subset G^0$.
\end{proposition}

\begin{proof}
Let $\fh$ denote the Lie subalgebra of $H$. By Corollary 4.25 of \cite{Knapp} we can decompose $\fh$ as $\fh=\fz(\fh)\oplus[\fh, \fh]$ where $\fz(\fh)$ is the center of $\fh$ and $[\fh, \fh]$ is semisimple. Since $H$ is connected and $\varphi$-invariant, we have that $\fh$ is $\DC$-invariant and since $\DC$ is a derivation it restricts to $\fz(\fh)$ and to $[\fh, \fh]$. 

The fact that $H$ is a compact subgroup together with the fact that $[\fh, \fh]$ is semisimple implies that the Cartan-Killing form restricted to $[\fh, \fh]$ is nondegene\-rated and negative definite (see \cite{ALEB} Chapter 4, Proposition 1.3). Since any derivation is skew symmetric by the Cartan Killing form we have that $\DC$ res\-tricted to $[\fh, \fh]$ must have only eigenvalues with zero real part which implies that $[\fh, \fh]\subset\fg^0$.

By Theorem 4.29 of \cite{Knapp} we have that the connected Lie subgroup associated with $\fz(\fh)$ is $Z(H)_0$ the connected component of the center of $H$. Since $Z(H)_0$ is connected, compact and abelian it is a torus and since the group of automorphisms of a torus is discrete we must have that $\varphi_t|_{Z(H)_0}=\id$ by continuity. Moreover, the fact that the exponential map of an abelian Lie group is a di\-ffeomorphism implies, by Proposition \ref{derivativeonorigin}, that $\DC(\fz(\fh))=0$ or that $\fz(\fh)\subset\ker\DC\subset\fg^0$ which concludes the proof.
\end{proof}

As a direct corollary we have:

\begin{corollary}
If $G$ is s compact Lie group and $\XC$ is a linear vector field on $G$, then the associated derivation $\DC$ has only eigenvalues with zero real part.
\end{corollary}

We will conclude the section with a technical lemma which will be used several times in the subsequent sections, its proof can be found in \cite{MW} Lemma 3.1.

\begin{lemma}
\label{salvador}
Let $G$ be a Lie group with Lie algebra $\fg$ and $N$ a normal subgroup of $G$ with Lie algebra $\fn$. Then, for any $X\in\fg$ we have that 
$$\exp(X+\fn)\subset\exp(X)N.$$
\end{lemma}

\subsection{Linear systems on Lie groups}

Let $\Omega$ be a subset of $\R^m$ such that $0\in\inner\Omega$ and consider the class of admissible control functions $\UC\subset L^{\infty}(\R, \Omega\subset\R^m)$. 

A {\bf linear system} on a Lie group $G$ is the family of differential equations
\begin{equation}
\label{linearsystem}
\dot{g}(t)=\XC(g(t))+\sum_{j=1}^mu_j(t)X_j(g(t)),
\end{equation}
where the drift vector field $\XC$ is a linear vector field, $X_j$ are right invariant vector fields and $u=(u_1, \cdots,u_m)\in\UC$. 

We will say that the system is {\bf bounded} if $\Omega$ is a compact convex subset of $\R^m$ and {\bf unbounded} if $\Omega=\R^m$.

\bigskip
The usual example of a linear system is the one on $\mathbb{R}^d$ given by
$$\dot{x}(t)=Ax(t)+Bu(t); \;\;A\in\mathbb{R}^{d\times d}, B\in\mathbb{R}^{d\times m}.$$
Since the right (left) invariant vector fields on the abelian Lie group $\mathbb{R}^d$ are given by constant vectors we can write the above system as 
$$\dot{x}(t)=Ax(t)+\sum_{i=1}^mu_i(t)b_i, \hspace{.5cm} B=(b_1| b_2| \cdots|b_m)$$
showing that it is a linear system in the sense of (\ref{linearsystem}). 
\bigskip

For a given $u\in\mathcal{U}$ and $t\in\mathbb{R}$ let us denote by $\phi_{t, u}:=\phi_{t, u}(e)$ the solution of (\ref{linearsystem}) starting at the neutral element $e\in G$. Then, for any $g\in G$ the solutions of (\ref{linearsystem}) starting at $g$ are given by
$$\phi_{t, u}(g)=\phi_{t, u}\cdot\varphi_t(g)=L_{\phi_{t, u}}(\varphi_t(g)),$$
where $L_h$ stands for the left translation on $G$ (see for instance Proposition 3.3 of \cite{DaSilva}). 
The map 
$$\phi:\mathbb{R}\times G\times\mathcal{U}\rightarrow G, \;\;\;(t, g, u)\mapsto\phi_{t, u}(g),$$
satisfies the {\bf cocycle property}
$$\phi_{t+s, u}(g)=\phi_{t, \Theta_su}(\phi_{s, u}(g)))$$
for all $t, s\in\mathbb{R}$, $g\in G$, $u\in\mathcal{U}$, where $\Theta_t:\UC\rightarrow\UC$ is the shift flow $u\in\UC\mapsto\Theta_tu:=u(\cdot +t)$. It follows directly from the cocycle property that the diffeomorphism $\phi_{t, u}$ has inverse $\phi_{-t, \Theta_{t}u}$ for any $t\in\R$ and $u\in\UC$. Also from the cocycle property and the fact that $\phi_{t, u}(g)$ just depends on $u|_{[0, t]}$ for any $t>0$ we have that 
$$\phi_{t, u_1}(\phi_{s, u_2}(g))=\phi_{t+s, u}(g)$$
where $u\in\UC$ is defined by $u(\tau)=u_1(\tau)$ for $\tau\in [0, s]$ and $u(\tau)=u_2(\tau-s)$ for $\tau\in[s, t+s]$. The function $u$ above is said to be the {\bf concatenation} of $u_1$ and $u_2$.

For any $g\in G$ the sets
\begin{equation}
\label{reachablesets}
\begin{array}{l}
\AC_{\tau}(g):=\{\phi_{\tau, u}(g), u\in\UC\}\\
\\
\AC(g):=\bigcup_{\tau>0}\AC_{\tau}(g),
\end{array}
\end{equation}
are the {\bf set of points reachable from $g$ at time $\tau$}  and the {\bf reachable set of $g$}, respectively. When $g=e$ is the neutral element of $G$ the sets $\AC_{\tau}(e)$ and $\AC(e)$ are said to be the {\bf reachable set at time $\tau$} and the {\bf reachable set} and are denoted, respectively, by $\AC_{\tau}$ and $\AC$.

When the system is bounded, the map $u\in\UC\mapsto\phi_{\tau, u}(g)\in G$ is continuous for any $\tau>0$ and $g\in G$. Consequently the sets $\AC_{\tau}(g)$ are compact sets (see for instance \cite{CKawan} Theorem 1.1).

We say that the system (\ref{linearsystem}) is {\bf controllable} if for some (and hence for all) $g\in M$ we have that $h\in\AC(g)$ and $g\in\AC(h)$ for all $h\in G$.

Let $\DC$ be the derivation associated with the linear vector field $\XC$ and let $\fh$ to be the smallest $\DC$-invariant Lie subalgebra of $\fg$ containing $X_i$, for $i=1, \ldots, m$. We say that the system (\ref{linearsystem}) satisfies the {\bf Lie algebra rank condition} if we have that $\fh=\fg$. By Krener's Theorem (see \cite{FCWK}, Theorem A.4.4), if the system satisfies the Lie algebra rank condition then the interior of $\AC_{\tau}(g)$ is nonempty for any $\tau>0$ and $g\in G$.

Let 
$$\dot{g}_j(t)=\XC^j(g(t))+\sum_{i=1}^mu_i(t)X^j_i(g(t)), \;\;\;\;u=(u_1, \ldots, u_m)\in\mathcal{U}$$
be linear systems on Lie groups $G_j$, $j=1, 2$ and consider a continuous map $\pi:G_1\rightarrow G_2$. We say that $\pi$ is a {\bf semi-conjugation} between the systems, if for any $g\in G_1$, $u\in\UC$ and $t\in\R$ we have
\begin{equation}
\label{semiconjugation}
\pi(\phi^1_{t, u}(g))=\phi^2_{t, u}(\pi(x))
\end{equation}
where $\phi^j$ denotes the respective solutions of the above systems. If $\pi$ is a homeo\-morphism we say that it is a {\bf conjugation} between the systems.

\begin{proposition}
\label{reachable}
It holds:
\begin{itemize}
\item[1.] if $0\leq \tau_1\leq \tau_2$ then $\AC_{\tau_1}\subset\AC_{\tau_2}$
\item[2.] for all $g\in G$ it holds that $\AC_{\tau}(g)=\AC_{\tau}\varphi_{\tau}(g)$;
\item[3.] for all $\tau, \tau'\geq 0$ we have $\AC_{\tau+\tau'}=\AC_{\tau}\varphi_{\tau}(\AC_{\tau'})=\AC_{\tau'}\varphi_{\tau'}(\AC_{\tau})$ and inductively that
$$\AC_{\tau_1}\varphi_{\tau_1}(\AC_{\tau_2})\varphi_{\tau_1+\tau_2}(\AC_{\tau_3})\cdots\varphi_{\sum^{n-1}_{i=1}\tau_i}(\AC_{\tau_n})= \AC_{\sum^n_{i=1}\tau_i}$$ 
for any positive real numbers $\tau_1, \ldots, \tau_n$;
\item[4.] for all $u\in\UC$, $g\in G$ and $t>0$ we have that $\phi_{t, u}(\AC(g))\subset\AC(g)$;
\item[5.] $e\in\inner\AC$ if, and only if, $\AC$ is open.
\end{itemize}
\end{proposition}

\begin{proof}
The proof of items 1. to 3.  can be found for instance in (\cite{JPh}, Proposition 2). 

For item 4. consider $h\in\AC(g)$. There are $s>0$ and $u'\in\UC$ such that $h=\phi_{s, u'}(g)$. If we consider $u''\in\UC$ to be the concatenation of $u'$ and $u$ we have that $\phi_{t, u}(h)=\phi_{t, u}(\phi_{s, u'}(g))=\phi_{t+s, u''}(g)\in\AC(g)$ showing that $\phi_{t, u}(\AC(g))\subset\AC(g)$.

The ``only if" part of item 6. is direct since $e\in\AC$. Let us then assume that $e\in\inner\AC$ and let $g\in\AC$. There exists $t>0$ and $u\in\UC$ such that $g=\phi_{t, u}$. By hypothesis we can find an open neighborhood $B$ of $e$ such that $B\subset\AC$. If we consider $C=\phi_{t, u}(B)$, we have that $C$ is an open neighborhood of $g$ and by item 5. above $C=\phi_{t, u}(B)\subset\phi_{t, u}(\AC)\subset\AC$ showing that $\AC$ is open.
\end{proof}

\begin{remark}
The third item of the above Proposition shows that the reachable set at time $\tau+\tau'$ is the product of the rea\-chable set at time $\tau$, by the image of the reachable set at time $\tau'$ by the flow $\varphi_{\tau}$.
\end{remark}

\begin{remark}
We should remark that the item 4. of the above proposition together with the fact that $0\in\inner\Omega$ shows us in particular that $\AC$ is invariant by $\varphi_t$ for any $t\geq 0$. 
\end{remark}

\section{The reachable set}

In this section we will show that under certain assumptions we have that the connected subgroup $G^{+, 0}$ is contained in the reachable set $\AC$. 

Let us assume from now on that $\AC$ is open\footnote{Remark 3.9 at the end of the section will give us an algebraic way to assure the openness of $\AC$.}. This condition implies, in parti\-cular, that the linear system satisfies the rank condition (see \cite{VAJT} Theorem 3.3).

We will start with some results concerning invariance.

\begin{lemma}
\label{pointinvariance}
Let $g\in\AC$ and assume that $\varphi_t(g)\in\AC$ for any $t\in\R$. Then $\AC\cdot g\subset\AC$.
\end{lemma}

\begin{proof}
Let $h=\phi_{t, u}\in\AC$. By hypothesis $\varphi_{-t}(g)\in\AC$ and there exists then $\tau>0$ such that $\varphi_{-t}(g)\in\AC_{\tau}$. Consequently
$$hg=h\varphi_t(\varphi_{-t}(g))\in \AC_t\varphi_t(\AC_{\tau})=\AC_{\tau+t}\subset\AC$$ 
which concludes the proof.
\end{proof}

As direct corollary we have the following:

\begin{corollary}
\label{invarianceprop}
If $H$ is a connected $\varphi$-invariant subgroup of $G$ with Lie algebra $\fh$ and $\exp X\in\AC$ for any $X\in\fh$ then $H\subset\AC$.
\end{corollary}

\begin{proof}
Since $H$ is connected and $\varphi$-invariant, its Lie algebra $\fh$ is $\DC$-invariant which implies that $\rme^{t\DC}X\in\fh$ for any $X\in\fh$ and any $t\in\R$. Consequently 
$$\varphi_t(\exp X)=\exp(\rme^{t\DC}X)\in\AC, \;\;\;\;t\in \R$$
and by the above lemma, any finite product of exponential of the elements of $\fh$ is in $\AC$ which implies that $H\subset\AC$.
\end{proof}

\begin{corollary}
\label{neighborhood}
If $H$ is a connected Lie subgroup and there exists a $\varphi$-invariant neighborhood of the neutral element $B\subset H\cap\AC$, then $H$ is $\varphi$-invariant and $H\subset \AC$.
\end{corollary}

\begin{proof}
Using Lemma \ref{pointinvariance} we have by the $\varphi$-invariance of $B$ that $B^n\subset\AC$ for any $n\in\N$. Since $H$ is connected we have that $H=\bigcup_{n\in\N}B^n$ which implies the result.
\end{proof}

The next result is a generalization of Proposition 5 of \cite{JPhDM}.

\begin{proposition}
\label{nilradical}
Let $\fh$ be a Lie subalgebra of $\fg$ and $\fn$ be an ideal of $\fh$ such that $\DC(\fh)\subset\fn$. If $N\subset\AC$, then $H\subset\AC$, where $N$ and $H$ are the connected Lie subgroups of with Lie algebras $\fn$ and $\fh$, respectively.
\end{proposition}

\begin{proof}
For any $X\in\fh$ and any $t\in\R$ we have that
$$\rme^{t\DC}X=X+\sum_{n\geq 1}\frac{t^n\DC^n}{n!}X=X+Y$$
where by the hypothesis $Y=Y_{t, X}\in\fn$. Since $\fn$ is an ideal, $N$ is normal and by Lemma \ref{salvador} we have that $\exp(X+Y)=(\exp X)g$ for some $g\in N$.

Let then $V$ to be an open neighborhood of $0\in \fg$ with $U=\exp(V)\subset\AC$ and such that $\exp|_{V}$ is diffeomorphism. Consider the open set of $H$ given by $B=U\cap H$. Since $H$ is connected we have that $H=\bigcup_{n\geq 1}B^n$. We will show inductively that $B^n\subset\AC$ which implies $H\subset\AC$.
If $n=1$ the construction of $B$ implies that $B\subset\AC$. Assume that $B^n\subset\AC$ and let $x\in B^{n+1}$. We can write $x=g_1\cdots g_{n+1}$ with $g_i\in B$, $i=1, \ldots n+1$. By the inductive hypothesis, we have that $h=g_1\cdots g_n\in \AC$. Let then $\tau>0$ such that $h\in\AC_{\tau}$. Since $g_{n+1}\in B$ there exists $\tau'>0$ with $g_{n+1}\in\AC_{\tau'}$ and we can write $g_{n+1}=\exp Z$ with $Z\in\fh$. By the above $\varphi_{\tau}(g_{n+1})=g_{n+1}g'$ with $g'\in N$. 
Since $\DC(\fh)\subset\fn$ we have in particular that $N$ is $\varphi$-invariant which implies that $g''=\varphi_{-\tau-\tau'}(g'^{-1})\in N\subset\AC$ and consequently that there is $\tau''>0$ such that $g''\in\AC_{\tau''}$. All together give us that
$$x=hg_{n+1}=h\varphi_{\tau}(g_{n+1})g'^{-1}=h\varphi_{\tau}(g_{n+1})\varphi_{\tau+\tau'}(g'')\in\AC_{\tau+\tau'+\tau''}\subset \AC$$
which concludes the proof.
\end{proof}

\subsection{Eigenvalues with zero real part}

Our aim here is to show that any solvable Lie subgroup of $G^0$ is contained in the reachable set. In order to do that the followind lemma will be central.

\begin{lemma}
\label{generalnilpotentcase}
Let $N\subset G^0$ to be a nilpotent connected $\varphi$-invariant Lie subgroup. Then $N\subset \AC$.
\end{lemma}

\begin{proof}
Let us denote by $\fn$ the Lie algebra of $N$ and consider its lower central series  
$$\fn=\fn_1\supset \fn_2\supset\ldots\supset \fn_{l+1}=\{0\}$$ where for $i=2, \ldots, l+1$ we have that $\fn_i=[\fn, \fn_{i-1}]$ are ideals of $\fn$.
Since $N$ is connected and $\varphi$-invariant its Lie algebra is $\DC$-invariant and consequently $\fn_i$ is $\DC$-invariant for $i=1, 2, \ldots, l+1$. We have then the decomposition
$$\fn_i=\bigoplus_{\alpha; \mathrm{Re}(\alpha)=0}\fn_{i, \alpha}\;\;\;\;\mbox{ where }\;\;\;\;\fn_{i, \alpha}=\fn_i\cap\fg_{\alpha}.$$

Moreover, $\fn_{i, \alpha}$ is given as union of $\DC$-invariant subspace as
\begin{equation}
\label{invsubspaces}
\fn_{i, \alpha}=\bigcup_{j\in\N_0}\fn_{i, \alpha}^j\;\;\;\;\mbox{ where }\;\;\;\;\fn^j_{i, \alpha}:=\{X\in\fn_{i, \alpha} ; \;(\DC-\alpha)^jX=0\}.
\end{equation}

For any $i=1, \ldots, l+1$ us denote by $N_i$ the connected subgroup of $N$ with Lie algebra $\fn_i$. We have that $N_i$ is a normal subgroup of $N$ and it is $\varphi$-invariant. We will divide the rest of the proof in 3 steps.

\begin{itemize}
\item[Step 1:]If for any $\alpha$ with $\;\mathrm{Re}(\alpha)=0$ and any $X\in\fn_{i, \alpha}$ it holds that $\exp X\in\AC$ then $N_i\subset\AC$;

Since $\fn_i$ is given as direct sum of the $\fn_{i, \alpha}$ with $\alpha$ as above, we have that the subset of $N_i$ given by
$$B:=\prod_{\alpha; \mathrm{Re}(\alpha)=0}\exp\fn_{i, \alpha}$$
is a neighborhood of the neutral element in $N_i$. Moreover, since
$$\varphi_t(\exp\fn_{i, \alpha})=\exp(\rme^{t\DC}(\fn_{i, \alpha}))=\exp \fn_{i, \alpha},\;\;\;\;t\in\R$$
$B$ is certainly $\varphi$-invariant and the hypothesis that $\exp X\in\AC$ for any $X\in\fn_{i, \alpha}$ together with Lemma \ref{pointinvariance} implies that $B\subset\AC$. Using Corollary \ref{neighborhood} we have then that $N_i\subset\AC$.

\item[Step 2:]If $N_{i+1}\subset\AC$ then $N_i\subset\AC$ for $i=0, \ldots l$;

By the above we just have to show that $\exp X\in\AC$ for any $X\in\fn_{i, \alpha}$ and any $\alpha$ with $\mathrm{Re}(\alpha)=0$. Using the decomposition (\ref{invsubspaces}) it is enough to prove that, for any $j\in\N_0$ and $X\in\fn_{i, \alpha}^j$ it holds that $\exp X\in\AC$. The last statement will be proved using inducion on $j\in\N_0$.

If $j=0$ we have that $\fn_{i, \alpha}^0=\{0\}$ and since $e\in\AC$ by hypothesis, the result holds. Let us assume then that for $j>0$ it holds that $\exp Z\in\AC$ for any $Z\in\fn_{i, \alpha}^{j-1}$ and let $X\in\fn_{i, \alpha}^j$. 
Since $\AC$ is an open neighborhood of $e\in G$ there exists $m\in\N$ such that for $v=\frac{X}{m}$ we have that $\exp v\in \AC$. Let then $\tau'>0$ such that $\exp v\in\AC_{\tau'}$. Also, since $\mathrm{Re}(\alpha)=0$, there exists $\tau''>0$ such that $\rme^{\tau''\alpha}=1$. Since $\rme^{n\tau''\alpha}=1$ for any $n\in\N$ let us consider $n_0\in\N$ such that $\tau=n_0\tau''\geq \tau'$. Then $\rme^{\tau\alpha}=1$ and $\exp v\in\AC_{\tau}$ by item 1. of Proposition \ref{reachable}.

Since $v\in \fn_{i,\alpha}^j$ we have for any $k\in\N$ that
$$\rme^{k\tau\DC}v=\rme^{k\tau(\DC-\alpha)}v=v+k\tau(\DC-\alpha)v+\cdots+\frac{(k\tau(\DC-\alpha))^{j-1}}{(j-1)!}v$$
 and consequently that 
\begin{equation}
\label{exponential}
v=\rme^{k\tau\DC}v+z_k
\end{equation}
where 
$$-z_k=k\tau(\DC-\alpha)v+\cdots+\frac{(k\tau(\DC-\alpha))^{j-1}}{(j-1)!}v\in \fn_{i, \alpha}^{j-1}.$$ 
Since $\fn$ is nilpotent we have by {\bf BCH} that
$$\exp(\rme^{\tau \DC}v)\exp(\rme^{2\tau \DC}v)=\exp(\rme^{\tau \DC}v+\rme^{2\tau \DC}v+O_1)$$
where $O_1$ is given by brackets of $\rme^{\tau \DC}v$ and $\rme^{2\tau \DC}v$. Since $[\fn^j_{i, \alpha}, \fn^j_{i, \alpha}]\subset [\fn, \fn_i]=\fn_{i+1}$ we have that $O_1\in\fn_{i+1}$ which by Lemma \ref{salvador} implies that 
$$\exp(\rme^{\tau \DC}v+\rme^{2\tau \DC}v+O_1)=\exp(\rme^{\tau \DC}v+\rme^{2\tau \DC}v)g_1$$
with $g_1\in N_{i+1}$. Since by Proposition \ref{reachable} item 3. we have that 
$$\exp(\rme^{\tau \DC}v)\exp(\rme^{2\tau \DC}v)=\varphi_{\tau}(\exp v)\varphi_{2\tau}(\exp v)\in\AC_{3\tau}\subset\AC$$
and $N_{i+1}$ is $\varphi$-invariant and is contained in $\AC$ by hypothesis, Lemma \ref{pointinvariance} implies that 
$$\exp(\rme^{\tau\DC}v+\rme^{2\tau \DC}v)\in\AC.$$

Using item 1. \hspace{-.2cm} of Proposition \ref{reachable} we can choose then $k_3\in\N$ such that
$$\exp(\rme^{\tau\DC}v+\rme^{2\tau \DC}v)\in\AC_{k_3\tau}.$$ Again by {\bf BCH}, we have that
$$\exp(\rme^{\tau \DC}v+\rme^{2\tau \DC}v)\exp(\rme^{k_3\tau \DC}v)=\exp(\rme^{\tau \DC}v+\rme^{2\tau \DC}v+\rme^{k_3\tau \DC}v+O_2)$$
where $O_2$ is given by brackets of $\rme^{\tau \DC}v+\rme^{2\tau \DC}v$ and $\rme^{k_3\tau \DC}v$ which implies as before that $O_2\in \fn_{i+1}$ and again by Lemma \ref{salvador} that 
$$\exp(\rme^{\tau \DC}v+\rme^{2\tau \DC}v+\rme^{k_3\tau \DC}v+O_2)=\exp(\rme^{\tau \DC}v+\rme^{2\tau \DC}v+\rme^{k_3\tau \DC}v)g_2$$
with $g_2\in N_{i+1}$. Since $\exp(\rme^{k_3\tau \DC}v)=\varphi_{k_3\tau}(\exp v)$ we have that 
$$\exp(\rme^{\tau \DC}v+\rme^{2\tau \DC}v)\exp(\rme^{k_3\tau \DC}v)\in\AC_{(k_3+1)\tau}\subset\AC.$$
and again by the $\varphi$-invariance of $N_{i+1}$ and the hypothesis that it is contained in $\AC$ we have by Lemma \ref{pointinvariance} that 
$$\exp(\rme^{\tau \DC}v+\rme^{2\tau \DC}v+\rme^{k_3\tau \DC}v)\in\AC.$$

If we repeat the same process as above $m-2$ times we get then $k_3, k_4,\ldots, k_m$ such that
$$\exp\left(\sum_{i=1}^m\rme^{k_i\tau \DC}v\right)\in\AC$$
with $k_1=1$ and $k_2=2$. Using (\ref{exponential}) and the fact that $v=\frac{X}{m}$ we have by summing up to $m$ that  
$$X=\sum_{i=1}^m\rme^{k_i\tau D}v +z$$
with $z\in\fn_{i, \alpha}^{j-1}$ and $k_i\in\N$ as above. Again by {\bf BCH}, we have that
$$\exp\left(\sum_{i=1}^m\rme^{k_i\tau \DC}v\right)\exp z=\exp (X+O)$$
where $O$ is given by brackets of $\sum_{i=1}^m\rme^{k_i\tau \DC}v$ and $z$ which as before implies that $O\in\fn_{i+1}$. By the $\DC$-invariance of $\fn_{i, \alpha}^{j-1}$ it holds that $\rme^{t\DC}z\in\fn_{i, \alpha}^{j-1}$ for any $t\in\R$ which by the inductive hypothesys implies that
$$\varphi_t(\exp z)=\exp(\rme^{t\DC}z)\in\AC,\;\;\mbox{ for any }\;\;t\in\R$$
and by Lemma \ref{pointinvariance} that 
$$\exp (X+O)=\exp\left(\sum_{i=1}^m\rme^{k_i\tau \DC}v\right)\exp z\in\AC.$$
Using now Lemma \ref{salvador} we have that 
$$\exp(X+O)=(\exp X)g, \;\;\;\mbox{ for some }\;\;\;g\in N_{i+1}$$
and again by the $\varphi$-invariance of $N_{i+1}$ and the hypothesis that $N_{i+1}\subset\AC$ we get that 
$$\exp X=\exp(X+O)g^{-1}\in\AC$$
which gives us the desired result.

\item[Step 3:] $N\subset \AC$. 

Since $N_{l+1}=\{e\}\subset\AC$ we just have to apply the above $l$-times to get that $N=N_1\subset\AC$ concluding the lemma. 
\end{itemize}
\end{proof}

The above together with Proposition \ref{nilradical} give us that any $\varphi$-invariant solvable connected Lie subgroup of $G^0$ is contained in $\AC$ as stated in the next result.

\begin{proposition}
\label{solvablecase}
If $K^0$ is a $\varphi$-invariant solvable connected Lie subgroup of $G^0$ it holds that $K^0\subset\AC$. In particular if $G^0$ is solvable we have that $G^0\subset\AC$.
\end{proposition}

\begin{proof}
Let $\fn^0$ to be the nilradical of $\fk^0$. If $N^0$ stands for its connected subgroup the above lemma implies that $N^0\subset\AC$. Since $K^0$ is connected we have that its Lie algebra $\fk^0$ is $\DC$-invariant and since it is solvable it holds that $\DC(\fk^0)\subset\fn^0$. By Proposition \ref{nilradical} we have then that $K^0\subset\AC$ which concludes the proof.
\end{proof}

We can now prove the statement made at the beginning of this section.

\begin{theorem}
\label{positivereachable}
Let $G$ to be a solvable connected Lie group and consider the linear system \ref{linearsystem} on $G$. If $\AC$ is open then $G^{+, 0}\subset\AC$.
\end{theorem}

\begin{proof}
Let $g\in G^+$. Since $G^+$ is nilpotent, there exists $X\in\fg^+$ such that $g=\exp X$. Since $0\in\fg^+$ is exponentially stable for the flow of $\rme^{t\DC}|_{\fg^+}$ in negative time, we have that $\rme^{-t\DC}X$ can be made arbitrarily close to $0\in\fg^+$ for $t>0$ large enough. By continuity we have that 
$$\varphi_{-t}(g)=\varphi_{-t}(\exp X)=\exp(\rme^{-t\DC}X)\in\AC$$
for $t>0$ large enough and consequently $g\in\varphi_t(\AC)\subset\AC$ showing that $G^+\subset\AC$. Since $G$ is solvable $G^0$ is also solvable which by Proposition \ref{solvablecase} implies that $G^0\subset\AC$. By Lemma \ref{pointinvariance} the product $G^{+, 0}=G^+G^0$ is then contained in $\AC$ concluding the proof. 
\end{proof}

\begin{remark}
It is important to notice that neither in Lemma \ref{generalnilpotentcase} nor in Proposition \ref{solvablecase} we assumed that the whole subgroup $G^0$ was solvable.
\end{remark}

\begin{remark}
The linear system (\ref{linearsystem}) is said to satisfy the {\bf ad-rank} condition if the vector subspace
\begin{equation}
\label{ad-rank}
\mathrm{span}\{\DC^j(X_i(e)); \;i=1, \ldots, m, j\in\N_0\}
\end{equation}
coincides with $\fg$. If the system satisfies (\ref{ad-rank}) then by Theorem 3.5 of \cite{VAJT} we have that $e\in\inner\AC_{\tau}$ for any $\tau>0$ which by Proposition \ref{reachable} implies, in particular, that $\AC$ is open. 

Then we have an algebraic method that implies the openness of reachable set $\AC$. Although for Euclidean system the ad-rank condition is equivalent to the set $\AC$ being open that is not true for any Lie group (see example in the last section).
\end{remark}

\section{Controllability of linear systems}

The results in the previous section will allow us to give sufficient conditions for the controllability of linear systems on solvable connected Lie groups. Let us assume from now on then that $G$ is a solvable connected Lie group.

Let $\XC$ to be a linear vector field on $G$ and $\DC$ its associated derivation. If we consider the linear vector field $\XC^*$ on $G$ whose flow is given by $\varphi_{\tau}^*:=\varphi_{-\tau}$ it is straightforward to see that the derivation $\DC^*$ associated with $\XC^*$ satisfies $\DC^*=-\DC$. The Lie subalgebras and Lie subgroups induced by the derivation $\DC^*$ are related with the ones induced by $\DC$ as
$$\fg_*^+=\fg^-, \;\;\;\;\;\fg^-_*=\fg^+,\;\;\; \mbox{ and }\;\;\;\fg^0_*=\fg^0$$
and
$$G_*^+=G^-, \;\;\;\;\;G^-_*=G^+,\;\;\; \mbox{ and }\;\;\;G^0_*=G^0.$$
Also, if we consider linear system (\ref{linearsystem}) with drifts $\XC^*$, $\XC$ and same right invariant vector fields, their respectives solutions are related by $\phi^*_{t, u}(g)=\phi_{-t, u}(g)$ which implies that $\AC^*_{\tau}=\varphi_{-\tau}((\AC_{\tau})^{-1})$. 

The above together with Theorem \ref{positivereachable} give us that if $\AC^*$ is open then $G^{-, 0}\subset\AC^*$. We have then the main result of the paper:

\begin{theorem}
\label{control1}
The linear system (\ref{linearsystem}) is controllable if $\AC$ and $\AC^*$ are open and the derivation associated with $\XC$ has only eigenvalues with zero real part.
\end{theorem}

\begin{proof}
By Theorem \ref{positivereachable} and the above discussion we have that $G^{+, 0}\subset \AC$ and $G^{-, 0}\subset \AC^*$. Since $\DC$ has only eigenvalues with zero real part, we have that $\fg=\fg^0$ which implies that $G=G^0$ and consequently that $G=\AC\cap\AC^*$. Let then $g\in G$ arbitrary. There is $\tau>0, u\in\UC$ such that $g=\phi^*_{\tau, u}=\phi_{-\tau, u}$. Since $\phi_{-\tau, u}=\phi_{\tau, \Theta_{-\tau}u}^{-1}$ we have that $\phi_{\tau, \Theta_{-\tau}u}(g)=e$ which implies that $e\in\AC(g)$ and show that the system (\ref{linearsystem}) is controllable.
\end{proof}

\begin{remark}
If the linear system is bounded, Lemma 4.5.2 of \cite{FCWK} implies that there exists $\tau_0>0$ such that $e\in\inner\AC_{\tau_0}$ and consequently that $\AC$ is open if, and only if $\AC^*$ is open. Moreover, if the system satisfies the ad-rank condition then we also have that $\AC$ and $\AC^*$ are open.
\end{remark}

\subsection{Bounded linear systems} 

From now on we will assume that the connected Lie group $G$ is nilpotent and that the linear system (\ref{linearsystem}) on $G$ is bounded. With such assumptions we will show that the conditions in Theorem \ref{control1} are also necessary. 

In order to prove the above we will need the following lemma:

\begin{lemma}
\label{normal}
Let us assume that $\AC$ is open and that $G=G^{-, 0}$. If the homogeneous space $M=G/G^0$ admits a $G$-invariant Riemannian metric then $\AC_{G^-}=\AC\cap G^-$ is a relatively compact set.
\end{lemma}

\begin{proof}
By the $\varphi$-invariance of $G^0$ we have a well induced flow $\Psi_t$ on $M$ that satisfies $\Psi_t\circ\pi=\pi\circ\varphi_t$ for any $t\in\R$, where $\pi:G\rightarrow M$ is the canonical projection. Let us denote by $\varrho$ the distance on $M$ induced by the $G$-invariant Riemannian metric on $M$. Then, for any smooth curve $\gamma:[0, 1]\rightarrow M$ joining two given points $x, y\in M$, we have that	 $\alpha(s):=\Psi_t(\gamma(s))$ is a smooth curve joining $\Psi_t(x)$ with $\Psi_t(y)$ and consequently  
$$\varrho(\Psi_t(x), \Psi_t(y))\leq\int_0^1|\alpha'(s)|ds=\int_0^1|(d\Psi_t)_{\gamma(s)}\gamma'(s)|ds.$$
Also by the $G$-invariance of the metric we have that $\mu_g:M\rightarrow M$ given by $\mu_g(x)=gx$ is an isometry and since $\Psi_t\circ\mu_g=\mu_{\varphi_t(g)}\circ\Psi_t$ we have that $||(d\Psi_t)_x||=||(d\Psi_t)_o||$ for any $x\in M$. By Proposition 3.5 of \cite{DaSilva} we have that $(d\Psi_t)_o=\rme^{t\DC|_{\fg^-}}$ which together with Remark \ref{exponentiallystable} gives us that 
$$\varrho(\Psi_t(x), \Psi_t(y))\leq c^{-1}\rme^{-\lambda t}\int_0^1|\gamma'(s)|ds, \;\;\mbox{ for any }\;\;t>0$$
and consequently, that
\begin{equation}
\label{metric}
\varrho(\Psi_t(x), \Psi_t(y))\leq c^{-1}\rme^{-\lambda t}\varrho(x, y)
\end{equation}
for any $t>0$ and $x, y\in M$. Consider then $\pi(\AC)\subset M$ and for any $t>0$ and $u\in\UC$ let us consider $\bar{\phi}_{t, u}:=\pi(\phi_{t, u})$. Since $\phi_{t+s, u}=\phi_{t, \Theta_su}\varphi_t(\phi_{s, u})$ we have that 
\begin{equation}
\label{cocycle}
\bar{\phi}_{t+s, u}=\mu_{\phi_{t, \Theta_su}}(\Psi_t(\bar{\phi}_{s, u})).
\end{equation}

By the compacity of $\UC$ and the continuity of $\bar{\phi}_{t, u}$ there exists $t_0>0$ such that $\bar{\phi}_{t, u}\in B_1(o)$ for any $t\in[0, t_0]$ and any $u\in\UC$. Let us assume w.l.o.g. that $c^{-1}\geq 1$. We will show by induction on $n\in\N$ that, for any $u\in\UC$, it holds that
$$\varrho(\bar{\phi}_{nt_0, u}, o)\leq c^{-1}\sum_{i=0}^{n-1}\rme^{-i\lambda t_0}.$$
By our assumptions we already have the case $n=1$. Let us assume then the result valid for $n\geq 1$ and let us show it for $n+1$. By equation (\ref{cocycle}), we have that  
$$\bar{\phi}_{(n+1)t_0, u}=\bar{\phi}_{nt_0+t_0, u}=\mu_{\phi_{nt_0, \Theta_{t_0}u}}( \Psi_{nt_0}(\bar{\phi}_{t_0, u}))$$
which gives us
$$\varrho(\bar{\phi}_{(n+1)t_0, u}, o)\leq \varrho(\mu_{\phi_{nt_0, \Theta_{t_0}u}}( \Psi_{nt_0}(\bar{\phi}_{t_0, u})), \mu_{\phi_{nt_0, \Theta_{t_0}u}}(o))$$
$$+\varrho(\mu_{\phi_{nt_0, \Theta_{t_0}u}}(o), o)=\varrho(\Psi_{nt_0}(\bar{\phi}_{t_0, u}), o)+\varrho(\bar{\phi}_{nt_0, \Theta_{t_0}u}, o)$$
where for the equality we used that $\mu_g$ is an isometry and that $\mu_g(o)=\pi(g)$ for any $g\in G$. By the inductive hypothesis and equation (\ref{metric}) we have that 
$$\hspace{-5cm}\varrho(\bar{\phi}_{nt_0, \Theta_{t_0}u}, o) \leq c^{-1}\sum_{i=0}^{n-1}\rme^{-i\lambda t_0}$$
$$\hspace{1cm}\mbox{and }\;\;\;\; \varrho(\Psi_{nt_0}(\bar{\phi}_{t_0, u}), o)\leq c^{-1}\rme^{-n\lambda t_0}\varrho(\bar{\phi}_{t_0, u}, o) \leq c^{-1}\rme^{-n\lambda t_0}$$
and consequently that
$$\varrho(\bar{\phi}_{(n+1)t_0, u}, o)\leq c^{-1}\sum_{i=0}^{n}\rme^{-i\lambda t_0}.$$
Since $\rme^{-\lambda t_0}<1$ we have that $r:=c^{-1}\sum_{i=0}^{\infty}\rme^{-i\lambda t_0}<\infty$ and by the above that 
$$\pi(\AC)=\pi\left(\bigcup_{n\in\N}\AC_{n\tau_0}\right)=\bigcup_{n\in\N}\pi(\AC_{n\tau_0}) \subset \bar{B}_r(o)$$
showing that $\pi(\AC)$ is a relatively compact set in $M$.
Since $G^{-, 0}=G^-G^0$ and $G^-\cap G^0=\{e\}$ we have that $\pi|_{G^-}$ is a homeomorphism between $G^-$ and $M$ which implies that $\AC_{G^-}$ is a relatively compact set in $G^-$ and since $G^-$ is closed it is a relatively compact set of $G$.
\end{proof}

We have then the following result.

\begin{proposition}
\label{reciprocanilpotent1}
Let us assume that $\AC$ is open. Then $\AC=G$ if, and only if, $G=G^{+, 0}$.
\end{proposition}

\begin{proof}
By Theorem \ref{control1} we have that $G^{+, 0}\subset\AC$ which implies that, if $G=G^{+, 0}$ then $\AC=G$. 

For the converse we will proceed by induction on the dimension of $G$. If $\dim G=1$ then $G$ is abelian and the system on $G$ is conjugated to its linearization by the exponential map. The result follows then from Theorem 6 of \cite{Sontag} for Euclidean systems.

Let us assume then that the result is true for any nilpotent connected Lie group with dimension smaller than $n$ and consider a Lie group $G$ in the conditions of the theorem with $\dim G=n$. Since the center of $G$ is nontrivial, we have that $H=G/Z(G)$ is a nilpotent connected Lie group with $\dim H<n$. Also, since $Z(G)$ is $\varphi$-invariant there is a well defined induced linear system on $H$ that is semi-conjugated to the system on $G$ by the canonical projection $\pi:G\rightarrow H$ (see Proposition 4 of \cite{JPh1}). 

Since $\AC=G$ we have that $\pi(\AC)=H$. Since the linear systems are semi-conjugated we have that $\pi(\AC)$ is the reachable of the linear system on $H$ and by the inductive hypothesis we must have $H=H^{+, 0}$. By Lemma \ref{projection} it implies then that $G^-\subset Z(G)$ and by item 5. of Proposition \ref{subgroups} that $G^+$ is a normal subgroup of $G$. Using again Proposition 4 of \cite{JPh1} we have an induced linear systems on $G/G^+=G^{-, 0}$ that is semi-conjugated to the system on $G$.

Let $\AC^{-, 0}$ denotes the reachable set of the induced system on $G^{-, 0}$. Since $G^-\subset Z(G)$ we have that $\Ad(G^0)|_{\fg^-}=\{\mathrm{id}_{\fg^-}\}$ is  compact in $\mathrm{Gl}(\fg^-)$ and consequently that $G^{-, 0}/G^0$ admits a $G$-invariant Riemannian metric (see Proposition 3.16 of \cite{JCDE}). Moreover, by the semi-conjugation between the systems we have that $\AC^{-, 0}=G^{-, 0}$ which by the above corollary implies that $\AC^{-, 0}\cap G^-=G^-$ is a relatively compact set of $G$. Since $G^-$ is closed it must be a compact subgroup which by Lemma \ref{compact} implies that $G^-$ must be trivial showing that $G=G^{+, 0}$ as desired.
\end{proof}

We can then prove the main result for bounded systems on nilpotent Lie groups.

\begin{theorem}
\label{reciprocanilpotent}
The bounded linear system (\ref{linearsystem}) on the nilpotent Lie group $G$ is controllable if, and only if, $\AC$ is an open set and $G=G^0$.
\end{theorem}

\begin{proof}
By Theorem \ref{control1} we just have to prove the ``only if" part. If the linear system is controllable then $\AC$ is certainly open and $G=\AC\cap\AC^*$. Applying Proposition \ref{reciprocanilpotent1} above to $\AC$ and to $\AC^*$ we get then the desired.
\end{proof}

\begin{remark}
The above theorem generalizes Theorem 6 of \cite{Sontag} for Euclidean systems. 
\end{remark}

\begin{remark}
We should emphasize that the nilpotency of $G$ in the above results was used mainly to assure that the center $Z(G)$ is not trivial and use it in our induction process. If the group $G$ is solvable, not necessarely nilpotent, and the subgroup $G^0$ is compact a similar induction process could be done in order to prove that the conditions of Theorem \ref{control1} are also necessary. 
\end{remark}

\subsection{Examples}
\begin{example}Unrestricted Linear systems on Nilpotent Lie groups

Although the converse holds for restricted system on nilpotente Lie groups, for unrestricted system it does not, as the next example shows (Example 7.1 of \cite{JPh}): 
Consider the nilpotent Heisenberg group
$$G=\left\{\left(\begin{array}{ccc}
1  &  y  &  z \\
0  &  1  &  x \\
0  &  0  &  1
\end{array}\right);\;\;\;x, y, z\in\R
\right\}.$$
Its Lie algebra is generated by the right invariant vector fields
$$X=\left(\begin{array}{ccc}
0  &  0  &  0 \\
0  &  0  &  1 \\
0  &  0  &  0
\end{array}\right),
\;\;
Y=\left(\begin{array}{ccc}
0  &  1  &  x \\
0  &  0  &  0 \\
0  &  0  &  0
\end{array}\right), \;\;
Z=\left(\begin{array}{ccc}
0  &  0  &  1 \\
0  &  0  &  0 \\
0  &  0  &  0
\end{array}\right),$$
satisfying $[X, Y]=Z$. Writing in natural coordinates we have
$$X=\frac{\partial}{\partial x}, \;\;\;\;Y=\frac{\partial}{\partial y}+x\frac{\partial}{\partial z}, \;\;\;\;Z=\frac{\partial}{\partial z}.$$

Associated with the derivation $\DC:\fg\rightarrow\fg$ defined as $\DC(X)=Y$, $\DC(Y)=X$ and $\DC(Z)=0$ we have the linear vector field
$$\XC=y\frac{\partial}{\partial x}+x\frac{\partial}{\partial y}+\frac{1}{2}(x^2+y^2)\frac{\partial}{\partial z},$$
that is, $\DC=\ad(\XC)$. As shown in \cite{JPh}, the linear system
$$\dot{g}=\XC(g)+u_1X(g)+u_2Z(g)$$
is controllable on $G$. However, it is easy to see that $X-Y$ and $X+Y$ are eigenvectors associated respectively with the eigenvalues $-1$ and $1$ and consequently neither $G^-$ nor $G^+$ are trivial.
\end{example}

\begin{example} Bounded linear system on solvable Lie group.

Although we do not know if Theorem \ref{reciprocanilpotent} also holds on solvable Lie groups without the assumption that the subgroup $G^0$ is compact, this example shows one case where it holds. 
Consider the connected Lie group
$$G=\left\{\left(
\begin{array}{cc}
x  &   y\\
0  &   1
\end{array}\right)
;\;\;\;(x, y)\in\R_{>0}\times\R
\right\}.$$

Its Lie algebra $\fg$ is the solvable right invariant Lie algebra generated by 
$$X=\left(
\begin{array}{cc}
-1  &   0\\
0  &   0
\end{array}\right) \hspace{2cm}
 Y=\left(
 \begin{array}{cc}
 0  &   1\\
 0  &   0
 \end{array}\right)
 $$
 with $[X, Y]=Y$.
 For any $g=\left(
 \begin{array}{cc}
 x  &   y\\
 0  &   1
 \end{array}\right)\in G$ we have that 
$$X(g)=\left(
\begin{array}{cc}
-1  &   0\\
0  &   0
\end{array}\right)g= \left(
\begin{array}{cc}
-x  &   -y\\
0  &   0
\end{array}\right)\hspace{1cm}
Y(g)=\left(
\begin{array}{cc}
0  &   1\\
0  &   0
\end{array}\right)g=\left(
\begin{array}{cc}
0  &   1\\
0  &   0
\end{array}\right).
$$ 
Consider the derivation $\DC$ that on the basis $\{X, Y\}$ has matrix given by 
$$\DC=\left(
\begin{array}{cc}
0  &   0\\
0  &   2
\end{array}\right),$$
that is, $\DC=2\ad(X)$. The related linear vector field is given by 
$$\XC(g)=2(gX-Xg)=2yY$$ 
for $g\in G$ as above. Consider then the linear system on $G$ given by
\begin{equation}
\label{afim}
\dot{g}=\XC(g)+uZ(g),
\end{equation}
where $u\in[-1, 1]$ and $Z=X+Y$. We have that
$$\mathrm{span}\{\DC^j(Z), j\in\N_0\}=\mathrm{span}\{Z, \DC(Z)\}=\mathrm{span}\{X+Y, 2Y\}=\fg$$
which implies that the above system satisfy the ad-rank condition (\ref{ad-rank}) and consequently that $\AC$ and $\AC^*$ are open sets. In coordinates the system is written as
$$\dot{x}=-ux\;\;\;\;\;\mbox{ and }\;\;\;\;\;\dot{y}=(2-u)y+u.$$

By our choices we have that

$$G^+=\left\{\left(
\begin{array}{cc}
1  &   y\\
0  &   1
\end{array}\right)
;\;\;\;y\in\R
\right\} \;\;\;\mbox{ and }\;\;\;G^0=\left\{\left(
\begin{array}{cc}
x  &   0\\
0  &   1
\end{array}\right)
;\;\;\;x\in\R_{>0}
\right\}.$$
Moreover $G^0$ is certainly a noncompact group, since it is homeomorphic to $\R_{>0}$. We will show that $\AC^*\neq G$ which implies that the system cannot be controllable. Since $u\in[-1, 1]$ we have that 
$$\dot{y}\leq 0\iff 0\geq (2-u)y+u\geq y+u\geq y-1\implies y\leq 1$$
therefore
$$y>1\implies \dot{y}>0$$
and the system cannot be controllable to $e=\left(
\begin{array}{cc}
1  &   0\\
0  &   1
\end{array}\right)$, that is, $\AC^*\neq G$ showing the desired.
\end{example}

\begin{remark}
The above example is a particular case of Theorem 3. of \cite{JPhDM} for unrestricted linear systems.
\end{remark}

\begin{example} The ad-rank condition.

This example shows that the condition on the openness of $\AC$ is weaker than the ad-rank condition:

Let $A, B\in sl(2)$ be given by
$$ A=\left(
\begin{array}{cc}
0  &   1\\
-1  &   0
\end{array}\right) \hspace{.5cm}\mbox{ and }\hspace{.5cm}
B=\left(
\begin{array}{cc}
-1  &   0\\
0  &   1
\end{array}\right).
$$
The linear vector field on the semisimple Lie group $Sl(2)$ given by $\XC(g):=Ag-gA$ has associated derivation given by $\DC=-\ad(A)$. Moreover, 
$$\DC(B)=\left(
\begin{array}{cc}
0  &   2\\
2  &   0
\end{array}\right) \hspace{.5cm}\mbox{ and }\hspace{.5cm}
\DC^2(B)=4B$$
which implies that the linear system
$$\dot{g}(t)=\XC(g(t))+u(t)Bg(t)$$
does not satisfy the ad-rank condition. However the above system is controllable which implies, in particular, that $\AC$ is open (for the details see Example 5 of \cite{JPh}).
\end{example}

\subsubsection*{Acknowledgments} I would like to thank Prof. Luiz San Martin and Dr. Christoph Kawan for helping with my doubts about Lie and Control Theory and also Prof. Philippe Jouan for discussing Example 2. above with me.

\end{document}